\documentclass[a4paper]{article}
\usepackage[OT2,T1]{fontenc}

\usepackage[russian,english]{babel}

\def\og{``}
\def\fg{''}

\newcounter{thenum}

\newenvironment{theorem}{\medbreak\refstepcounter{thenum}\textsc{Theorem} %
\thethenum. ---  \em  }{\rm\smallbreak}

\newenvironment{definition}{\medbreak\refstepcounter{thenum}
\textsc{Definition} \thethenum. ---  \em  }{\rm\smallbreak}

\newenvironment{remark}{\medbreak\refstepcounter{thenum}{\em Remark} %
\thethenum. --- }{\smallbreak}

\newenvironment{example}{\medbreak\refstepcounter{thenum}{\em
    Example} %
\thethenum. --- }{\smallbreak}

\newenvironment{proof}{\unskip\smallbreak{\sc Proof.} --- \rm}{\quad\bull\smallskip\rm}

\def\cI{{\cal I}}
\def\cJ{{\cal J}}
\def\bJ{{\bf J}}
\def\Jm{J\hbox to 0pt{$\setminus$\hss}}

\def\cO{{\cal O}}
\def\cP{{\cal P}}

\def\N{{\bf N}}
\def\R{{\bf R}}

\def\ord{{\rm ord}}

\def\bull{\vrule height .9ex width .8ex depth -.1ex}

\def\pa{\phantom{^{\ast}}}

\begin{document}

\title{Jacobi's bound and normal forms computations.\\ 
A historical survey}
\author{ F.~Ollivier, LIX UMRS CNRS--\'Ecole polytechnique n$^{\rm o}$~7161\\ 
    Email~\texttt{francois.ollivier@lix.polytechnique.fr.}\\
  \texttt{http://www.lix.polytechnique.fr/\lower0.8ex\hbox{\~{ }}ollivier/indexEngl.htm}}

\maketitle

{\narrower\it\noindent This work is dedicated to the memory of
Giuseppa Carrà Ferro and\hfill\break Evgeni{\u\i} Vasil'evich
Pankratiev.

}
\bigskip\bigskip

\begin{abstract}
\noindent Jacobi is one of the most famous mathematicians of his
century. His name is attached to many results in various fields of
mathematics and his complete works in seven volumes have been
available since the end of the \textsc{xix}$^{\rm th}$ century and are
very often quoted in many papers. It is then surprising that some of
his results may have fallen into oblivion, at least in part. We will
try to describe some of Jacobi's results on ordinary differential
equations and the available, published or unpublished material he
left. We will then expose the selective interests of his followers and
their own contributions.

There are in fact many interrelated results: a bound on the order of a
differential system, a necessary and sufficient condition, given by a
determinant, for the bound to be reached, an algorithm to compute the
bound in polynomial time, and processes for computing normal forms
using as few derivatives as possible.

We give for all of them the form under which they could have been
proved or rediscovered, sometimes independently of Jacobi's findings. In
conclusion, we give the state of the art and suggest some possible
applications of Jacobi's bound to improve some algorithms in
differential algebra.

\end{abstract}


\section{Introduction} In two posthumous articles
\cite{Jacobi1,Jacobi2}, which have been recently
translated\cite{Jacobi1e,Jacobi2e}, Jacobi 
has introduced a bound on the order of a system of $n$ ordinary
differential equations in $n$ unknowns. Let $A:=(a_{i,j})$ be the
matrix such that $a_{i,j}$ is the order of the equation $u_{i}$ in the
unknown function $x_{j}$. Let $J=\max_{\sigma\in S_{n}}
\sum_{i=1}^{n} a_{i,\sigma(i)}$. A sum $\sum_{i=1}^{n}
a_{i,\sigma(i)}$ is called a \emph{transversal sum} and $J$ is the
\emph{maximal transversal sum }. He claims that:
\medskip

\emph{\textsc{Jacobi's bound}. --- The order of the system
is bounded by $J$.}
\smallskip

\noindent The bound is still conjectural in the general
case.

\emph{\textsc{Jacobi's algorithm}. --- Jacobi gave an algorithm to
compute the bound in polynomial time, viz $\cO(n^{3})$ operations,
instead of trying the $n!$ permutations.}
\smallskip

\noindent It has been forgotten and rediscovered by Kuhn
in 1955  \cite{Kuhn55}, using Egerv\'ary's results (see Schrijver's
paper \cite{Schrijver05} for historical details).  The idea is to find
a \emph{canon}, i.e.\ $\lambda\in\N^{n}$ such that, in the matrix
$(a_{i,j}+\lambda_{i})$ one can select maximal entries in each column
that are located in different rows. Jacobi's algorithm computes the
unique canon with a minimal $n$-uple of integres $\lambda_{i}$ that we
denote by $\ell$. Let $\Lambda=\max_{i}
\ell_{i}$, $\alpha_{i}=\Lambda -\ell_{i}$ and
$\beta_{j}=\max_{i} a_{i,j}-\alpha_{i}$. The \emph{truncated jacobian
matrix} $\nabla$ is the matrix $\left(\partial u_{i}/\partial
x_{j}^{(\alpha_{i}+\beta_{j})}\right)$.
\medskip

\emph{\textsc{The truncated determinant condition}. --- Jacobi claims that the
order of the system is equal to the bound $J$ iff $|\nabla|\neq 0$.}
\smallskip

\noindent This implicitly assumes the \emph{strong bound}, defined with the
convention $\ord_{x_{j}} u_{i}=-\infty$ if $u_{i}$ is free of $x_{j}$
and its derivatives, but he gives no detail about what should be done
in such a case. 
\medskip

\emph{\textsc{The shortest reduction
method}. --- Jacobi also asserts that 
  it is possible to compute a normal form using only $\ell_{i}$
  derivatives of equation $u_{i}$ and that it is impossible to
  compute one using a smaller number of derivatives.}
\smallskip

\noindent This implicitly assumes $|\nabla|\neq0$; if not, a greater
  number of derivatives may be required. This is only generically true:
  for some particular systems, it is possible to differentiate $u_{i}$ at
  most $\ell_{i}-\ell_{i+1}$ times, for $i<n$, assuming
  $\lambda_{1}\ge\cdots\ge\lambda_{n}$). Jacobi also gives a bound on
  the order of derivation of the $u_{i}$ required to compute a
  resolvent representation, using $x_{j}$ as a differential primitive
  element---assuming it is one. {\em Then, $u_{i}$ must be differentiated a
  number of times equal to the maximal transversal sum of the matrix
  obtained by suppressing the row $i$ and the column $j$ in $A$.}
\bigskip

The aim of this paper is to describe the content of Jacobi's two
papers  \cite{Jacobi1,Jacobi2}, the genesis of
these results, the history of research on the subject and the state of
the art. We also describe some related documents from \emph{Jacobis
  Nachla{\ss}}, kept in the \emph{Archiv der Berlin-Brandenburgische
  Akademie der Wissenschaften} and give a complete\footnote{For the
  best of our knowledge\dots\ any information about material or
  sources not mentionned here is welcome.} list of references.

\section{Unpublished manuscripts} Jacobi himself is
possibly the first to have forgotten his own work. According to
Koenigsberger  \cite{Koenigsberger1904}, his manuscripts
on this subject were written around 1836 and were intended to be a
part of a forsaken project of a great work on differential
equations. Part of it was incorporated in his long paper on
the last multiplier \cite{Jacobi3}, but the bound itself was never published in his
lifetime. The many versions of the text, containing numerous
corrections, suggest that Jacobi was not satisfied of the
redaction. However, these manuscripts were clearly intended for
publication at the time he wrote them, as it is suggested by the many
typographical precisions in German, written in \emph{Kurrentschrift}
in the margins. In his preface to [VD], were the second article \cite{Jacobi2}
was first published, Clebsch asserts that it is a little
posterior to Jacobi's lectures at K{\H o}nigsberg university during
academic year 1842--43, when Borchardt was his student there.

On page~2231 of manuscript II/23 a), one finds a reference to a paper
\cite{Jacobi4} published in 1834 and his work on normal form
computation is mentionned in the second part of the last multiplier
paper \cite{Jacobi3}, published in 1845, but one may think that the
results were obtained before the redaction of that paper; the first
page of the manuscript I/58 a) shows that Jacobi began to correct it
in Roma, in december 1843.  Assuming that manuscripts {II/13 b),
II/22, II/23 a), II/23 b)} where written at the same time, as many
similarities of style and content suggest, it must have been between
1834 and the end of 1843, certainly not later than 1845.
\medskip

These results are a by-product of his work on the
\emph{isoperimetric problem}: {\em ``Let $U$ be a given
function of the independent variable $t$, the dependent ones $x$, $y$,
$z$ etc.\ and their derivatives $x'$, $x''$, etc., $y'$, $y''$, etc.,
$z'$, $z''$, etc.\ etc.\ If we propose the problem of determining the
functions $x$, $y$, $z$ in such a way that the integral
$$
\int U dt
$$
be \emph{maximal or minimal} or more generally that the differential of
this integral vanishes, it is known that the solution of the problem
depends on the integration of the system of differential equations:
\def\cale{\vbox to 6pt{\vfill}\vtop to 2pt{\vfill}}
$$\begin{array}{rcl} 
0 &=&\frac{\cale \partial U}{\cale \partial x} -
\frac{d\frac{\cale \partial U}{\cale \partial x'}} 
  {\cale dt} + \frac{d^{2}\frac{\cale \partial U}{\cale \partial
x''}}{\cale dt^{2}} - \hbox{etc.,}\\ 
0 &=&\frac{\cale \partial U}{\cale \partial y} -
\frac{d\frac{\cale \partial U}{\cale \partial y'}} 
  {\cale dt} + \frac{d^{2}\frac{\cale \partial U}{\cale \partial
y''}}{\cale dt^{2}} - \hbox{etc.,}\\ 
0 &=&\frac{\cale \partial U}{\cale \partial z} -
\frac{d\frac{\cale \partial U}{\cale \partial z'}} 
  {\cale dt} + \frac{d^{2}\frac{\cale \partial U}{\cale \partial
z''}}{\cale dt^{2}} - \hbox{etc.\ etc.} 
\end{array}
$$
I will call these in the following \emph{isoperimetric differential
  equations} \dots''\/} (see Jacobi's last multiplier article
\cite{Jacobi3}, GW~IV p.~495).  

If the highest order derivative of $x_{i}$ in $U$ is
$x_{i}^{(e_{i})}$, the order of $x_{j}$ in the $i^{\rm th}$
isoperimetric equation is $e_{i}+e_{j}$. Then, if the $e_{i}$
are not all equal to the maximal order $e:=\max_{i} e_{i}$, we cannot
compute a normal form without using \emph{auxiliary equations}
obtained by differentiating the $i^{\rm th}$ isoperimetric equation
$\lambda_{i}$ times with $\lambda_{i}=e-e_{i}$. It is also clear that
$J=2\sum_{i} e_{i}$ is equal to the order of the system, provided that
the Hessian matrix $(\partial^{2}U/\partial x_{i}\partial x_{j})$ has
full rank. We understand how this example can have inspired the whole
theory.
\medskip

In a letter to his brother Moritz, on Sept. 17$^{\rm th}$ 1836
\cite{Briefwechsel}, Jacobi writes 
about a huge manuscript on mechanics: \emph{``I came accross some very
abstract ideas about the treatment of differential equations that
appear in mechanical problems, for these differential equations, with
their special form, allow some simplifications for the integration,
that had not yet been remarked. These considerations will be all the
more important, I think, as they extend to the differential equations
that appear both in the isoperimetrical problem and the integration of
partial differential equations of the first order.''}  It seems that
the remaining manuscripts come from the time of these first
investigations, so between  1836 and 1840, rather than from some later attempt. The
tables $A$--$H$ given in  \cite{Jacobi2} are written on page 2250.a of
the manuscript;
on the back there is a table of the doubles of prime numbers equal to
$1$ modulo $8$, from $2018$ to $20018$, a material that could reflect
the strong interest of Jacobi in number theory and prime numbers, a
short time before the publication of Jacobi's \emph{Canon arithmeticus}
 \cite{Canon}, a table of discrete logarithms, in 1839. In a letter to
the Académie des Sciences de Paris, published by Liouville in
1840 \cite{AcademyLetter}, he evocates his work on mechanics. Liouville's
commentaries are enthousiastic and suggest that a book will appear
soon. But Jacobi writes to his brother in january
1841 \cite{Briefwechsel} that he is embarrassed, as he
does not have enough ``breath'' to achieve his huge project, that should
have been entitled \emph{Phoronomie}, the study of
physical bodies motion.
\smallskip

But in 1845, Jacobi had clearly still in mind to publish a study on normal
forms computation, for he wrote: \emph{``I will expose in another paper the
various ways by which this operation may be done, for this question
requires many outstanding theorems that necessitate a longer
exposition.''} \cite{Jacobi3} 
\medskip

It is quite possible that this project was forgotten because of a
change in Jacobi's life---who definitely left Koenigsberg to Berlin
after a long trip in Italy---that also opened new contacts and new
scientifitic issues. One may also consider a possible lack of
practical examples for such a general method of computing normal
forms. The algorithm may have suffered the same absence of contemporary
applications.  Jacobi was right claiming that the problem of computing
the bound was of interest by itself, but the economical questions that
strongly motivated the mathematicians of the last century were not yet
considered in the middle of the \textsc{xix}$^{\rm th}$ century. (See
section \ref{assignment}.)

\section{The publication of the manuscripts} 

Jacobi's widow gave the manuscripts he left to Dirichlet who
began to work for their publication with his friends
Borchardt and Joachimsthal. Very few documents
remain from their work and the best source seems to be Koenigsberger
 \cite{Koenigsberger1904}. The papers were in great
disorder. In order to class them, they gave a number to each
page. These numbers appear on the envelops were pages that seemed to
form a single document or to be related were stored.

Borchardt entrusted Sigismund Cohn\footnote{Almost nothing is known
about him. There was a student of that name at Könisberg
university in 1842-43; in 1846 some Sigismundus Cohn defended in
Breslau an inaugural dissertation entitled \emph{De medicina
talmudica}. According to the \emph{vita} that follows this work in the
copy kept at the library of Alliance Israélite Universelle, he was a different man with no interest
in mathematics.}, who worked on the publication of some others
manuscripts of Jacobi, with the documents related to the bound. Cohn
indentified (see II/13 a) two sets of manuscripts suitable for
publication
II/13 b), II/23 b) and worked on a transcription {II/13 c)} of
these sometimes hardly readable texts. After his death in 1861 \cite{Jacobi5}, the
work was continued by Borchardt who published the first paper
 \cite{Jacobi1} in his journal in 1865. The second  \cite{Jacobi2} was
published by Clebsch in the volume \emph{Vorlesungen über Dynamik}
[VD] in 1866. This one was quoted by Sofya Kovalevskaya in one
of her most famous articles \cite{Kovalevskaya1875} in 1875. The fact
that these papers were written in latin did not seem to have been a
trouble at that time. Cohn and Borchardt could easilly write themselves
some paragraphs to fill gaps in the manuscripts and Borchardt even
tried to rewrite full passages in order to make them clearer. In his
biography  \cite{Koenigsberger1904}, published in 1904, Koenigsberger
did not translate the many quotations in Latin, French and Italian.

Borchardt also wrote some kind of abstracts of the two papers, which
show that he fully understood their content and that he did not
consider what he published as devoid of rigor. A slightly ironical
quotation of Jacobi himself \emph{``Tam quaestiones altioris indaginis
poscuntur.''}\footnote{Then these questions require further investigation.}
concludes the abstract of a part that clearly did not satisfy
Borchardt's standards and was not kept in the published version (see
II/13 c).

\section{Jacobi's mathematical results}

\subsection{The manuscripts}
The first paper  \cite{Jacobi1} begins with the exposition of the bound
and the truncated determinant criterion. Then, the main part of the
paper is devoted to a careful explanation of the algorithm with
complete proofs. To obtain this paper, Cohn put together two different
texts from document II/13 b). The first $6$ pages of the
manuscript, reproduced in Cohn's transcription, were not kept by
Borchardt (see sec.~3). They are related to the different normal forms that a given
system could have, with a quite complete description of systems of $2$
equations in $2$ differential unkowns.

The second  \cite{Jacobi2} begins with a fast exposition of the
algorithm, without any proof, followed by the example of a $10\times
10$ matrix. Jacobi then explains how to compute a normal form, using as
few derivatives of each variable as possible and how to compute a
resolvent representation for some variable $x_{\kappa}$, using
again as few derivatives of each equation as possible. This paper
reproduces with very few changes a single manuscript {II/23 b)}.
\medskip

Cohn considered documents {II/22, II/23 a)} as unusable because they
investigate how Jacobi's last multiplier behaves when one changes the
order on derivatives and computes a new normal form. 
Borchardt and him wanted to avoid the multiplier theory, possibly
because Jacobi decided not to include this material in his paper on
the subject. The \S~17 in manuscript II/23~a) fos 2217--2220
corresponds to the same paragraph in the last multiplier paper
\cite{Jacobi3}\footnote{See Crelles~29, Heft~3 221--225 or GW~IV
403--407.}, but the manuscript considers the general situation, whereas
Jacobi retreated to the linear case in the published version.

Document {II/4} is not related to the bound or normal form
computation, but rather to the last multiplier theory. It seems 
an interesting unpublished paper of Jacobi on differential
equations, including results such as the linear independence criterion
given in Ritt's {\em Differential algebra\/} \cite{Ritt50} p.~34.

\subsection{The algorithm}\label{algo} A square table $(a_{i,j})$ being given,
Jacobi calls \emph{transversal maxima} numbers being maximal in
their column, that are located in all different rows. The idea of
jacobi for computing $\max_{\sigma\in S_{n}}
\sum_{i=1}^{n}a_{i,\sigma(i)}$ is to look to what he calls a
\emph{canon}\footnote{The choice of this strange word may be related
to its use in the title: \emph{Canon arithmeticus}, that according
to Schumacher came from a play on words: the computations were done by
a \emph{Kanonier Unteroffizier}. See Jacobi's {\em
Briefwechsel\/} \cite{Briefwechsel} note~2 p.~62.}, that is a $n\times
n$ square table where one finds a maximal set of $n$ transversal
maxima. One starts with a table $(a_{i,j})$ and computes a canon by
adding to all the elements of row $i$ some integer $\ell_{i}$. At
each step of Jacobi's algorithm, one tries to increase the number
of transversal maxima. The integers $\ell_{i}$ computed by the
algorithm are the smallest ones, meaning that there is no canon
derived from the table $a_{i,j}$ such that one of these integers may
be smaller.
\medskip

The algorithm may be sketched as follows. We start with the \emph{
preparation process}: we add suitable integers to the rows of the
table, so that each row possess a \emph{maximum}, i.e.\ an element being
maximal in its column. 
\smallskip

The second part of the algorithm starts with the \emph{prepared table}
the maximum of the first row being the first set of transversal
maxima being considered. 

A set of transversal maxima being given, we will repeatedly compute a
new set containing one more element. We reorganise the rows and
columns in the following way: \emph{upper rows} are the rows
containing the transversal maxima and \emph{lower rows} the remaining ones;
\emph{left columns} are the columns containing the transversal maxima and
\emph{right columns} the remaining ones. We denote by asterisks the
transversal maxima and the maxima (in their columns) that are located
in the upper rows and right columns: maxima with an asterisk are the
\emph{stared maxima}. If there is some maximal element in some lower
row and right column, we may already add it to the set of transversal
maxima.
\smallskip

We say that \emph{there is a path} (\emph{transitum dari}) from
row $i$ to row $j$ if some element of row $j$ is equal to a stared
maximum in $i$, or if there is a path from row $i$ to row $i'$ and
from $i'$ to $j$. Rows of the \emph{first class} are the upper rows
containing a stared element in a right column and all the rows to
which there is a path from them.  If there is a lower row $j$ in the
first class, there is a path from an upper row $i_{0}$ with a stared
maximum $\alpha_{0}$ in a right column to an upper row $i_{1}$
possessing an element $\alpha_{1}$ equal to the transversal maxima of
$i_{0}$, then a path from row $i_{1}$ to row $i_{2}$ possessing an
element $\alpha_{2}$ equal to the transversal maxima of $i_{1}$, then
a path from row $i_{2}$ to row $i_{3}$\dots\ and at the end a path from
an upper row $i_{r-1}$ to the lower row $i_{r}=j$ containing an
element $\alpha_{r}$ equal to the transversal maxima in $i_{r-1}$. We
may then replace in the set of transversal maxima those located in the
rows $i_{0}$, \ldots, $i_{r-1}$ by $\alpha_{0}$, \ldots,
$\alpha_{r-1}$ and get a greater set of transversal maxima by adding
$\alpha_{r}$. \emph{ There is no lower series of the first class iff
the set of transversal maxima is maximal.}

If so, we will need to increase some rows. The rows not in the first
class, from which there is a path to a lower row form the \emph{third
class}. The \emph{second class} contains the remaining rows. We will
increase all the elements of the third class rows by the minimal
integer such that one of them become equal to a stared element $\beta$
in a row of the first or second class.  The computations of the
maxima in each column and of the partition of the rows in classes may
be done in $\cO(n^{2})$ operations.
\smallskip

If the element $\beta$ belongs to the second class, its rows goes to
the third, as well as all the rows from which there is a path to it,
so that the change in the classes partition may be computed in
$\cO(n)$. This may happen at most $\cO(n)$ times before there is no
more elements in the second class. If $\beta$ belongs to the first
class, some lower series of the third class will go to the first, and the
number of transversal maxima will be increased, as we have seen. So,
we need to perform the partition in classes and exhaust the second
class, with total cost $\cO(n^{2})$, at most $n$ times
before we get a maximal set of $n$ transversal maxima. The complexity
of the whole algorithm is $\cO(n^{3})$.
\medskip

Basically, this algorithm is the same as Kuhn's ``Hungarian method''
(see section \ref{assignment}), but Kuhn also adds constants to the
columns, and not only to the rows. The complexity of Jacobi's
algorithm, $\cO(n^{3})$, is the same as the complexity of the variant
of Kuhn's method given by Munkres \cite{Munkres57}.

Jacobi proved further that the $\ell_{i}$ are minimal. This result is
a consequence of the two following propositions: \emph{i) there is no \emph{
unchanged} row of the third class, i.e.\ with $\ell_{i}=0$}; \emph{ii) the
numbers added to the third class series in the algorithm are the
minimal ones that may change the partitions in classes}.
\smallskip

On may remark that this is the only place in the two posthumous papers
 \cite{Jacobi1,Jacobi2}\footnote{Together with a short passage in the
second paper
 \cite{Jacobi2} related to resolvent computation and using the same
kind of combinatorial arguments. See subsection \ref{resolvent}.} where Jacobi
provides complete proofs of his results.

\subsection{The shortest reduction in normal form} 

Jacobi provides \cite{Jacobi2} a method to compute a normal form,
using as few derivatives as possible of the given equations
\begin{equation}
\label{systeme}
u_{1}(x_{1}, \ldots, x_{n})=0, \ldots, u_{n}(x_{1}, \ldots, x_{n})=0.
\end{equation}
His results are only \emph{generically}
true; it is easy, as we will see, to make the requested hypotheses
explicit. First, he implicitly assumes that the truncated determinant does not
vanish. In order to provide rigorous statements, we need to translate
his findings within the framework of some formalism that did
not exist at this time. We will use here an elementary approach that
relies on  diffiety theory  \cite{OllivierSadik2006a}.

\begin{definition}
\label{truncated} For short, we say that $\lambda_{i}$ is a
\emph{canon} for the order 
matrix $(a_{i,j})$, if $(a_{i,j}+\lambda_{i})$ is a canon, i.e.\ a
table possessing a maximal set of transversal maxima.

Let $\ell_{i}$ be the minimal canon of the order matrix
$(a_{i,j})$\footnote{That we may compute using Jacobi's algorithm, see
subsection \ref{algo}.},
with $a_{i,j}=\ord_{x_{j}} u_{i}$, $\Lambda=\max_{i}
\ell_{i}$, $\alpha_{i}=\Lambda -\ell_{i}$ and
$\beta_{j}=\max_{i} a_{i,j}-\alpha_{i}$. The \emph{truncated jacobian
matrix} $\nabla_{u}$ is the matrix $\left(\partial
u_{i}/\partial x_{j}^{(\alpha_{i}+\beta_{j})}\right)$.

We will say that an ordering $<$ on derivatives is a \emph{Jacobi
  ordering} for the system \ref{systeme} if
  $k_{1}-\beta_{1}<k_{2}-\beta_{2}$ implies
  $x_{i_{1}}^{(k_{1})}<x_{i_{2}}^{(k_{2})}$.

Jacobi calls an \emph{explicit canonical form} a system
$$
x_{i}^{(e_{i})} =f_{i}(x),\quad 1\le i\le n
$$
where the functions $f_{i}$ only depend on derivatives of $x_{j}$
smaller than $e_{j}$.
\end{definition}

Jacobi's shortest reduction method may be expressed in the following way.

\begin{theorem}\label{shortest-reduction} i) Assume there exist some functions $\tilde X:t\mapsto \tilde
  x_{j}(t)$, $[a,b]\mapsto \R$ that form a solution of the system $u$,
  such that $|\nabla_{u}|\neq0$. Then, there exists a normal form of
  $u$ for a Jacobi ordering, of which $\tilde x$ is solution, that may
  be computed using equation $u_{i}=0$ and its derivatives up to the
  order $\ell_{i}$.

ii) Assume that $\nabla_{u}$ and all its minors of order $n-1$ that do
not contain row $\mu$ have full rank, then there is no normal form of
$u$ that may be computed using only derivatives of $u_{\mu}$ up to an
order strictly less than $\ell_{\mu}$.
\end{theorem}
\begin{proof} i) We may reorder the unknowns $x_{j}$ and the equations
  $u_{i}$, so that the sequence $\alpha_{i}$ of
  definition~\ref{truncated} is increasing and the $n$ principal
  minors of $\nabla_{u}$ have non vanishing determinants. Let $D_{k}$
  be the determinant of the $k^{\rm th}$ minor. Consider the jacobian
  matrix of the system $\{u_{i}^{(k)}|1\le i\le n,\>0\le
  k\le\ell_{i}\}$, with respect to the set of derivatives
  $E:=\{x_{i}^{(\alpha_{i}+\beta_{i}+k)}|1\le i\le n,\>0\le
  k\le\ell_{i}\}$. Its determinant is a product of powers of the
  $D_{i}$ and so is not vanishing. We can then use the implicit
  functions theorem and find, on a suitable open set containing
  $(\tilde x_{1}(0)$, \dots, $\tilde
  x_{1}^{(\alpha_{1}+\beta_{1}+\ell_{1})}$, \dots,
  $\tilde x_{1}(0)$, \dots, $\tilde x_{n}^{(\alpha_{n}+\beta_{n}+\ell_{n})})$,
functions 
  expressing the derivatives of $E$, depending on the derivatives of
  the set $S:=\{x_{i}^{(k)}|1\le i\le n,\>0\le
  k<\alpha_{i}+\beta_{i}\}$. Hence we get a normal form:
$$
x_{i}^{(\alpha_{i}+\beta_{i})}= f_{i}(x),
$$
of which $\tilde X$ is a solution.
The order associated to this normal form is a Jacobi ordering.

ii) Assume that there exists a normal form that may be computed using
derivatives of $u_{i}$ up to order $\lambda_{i}$, and that
$\lambda_{\mu}<\ell_{\mu}$. Let $s:=\max_{i=1}^{n}
\lambda_{i}-\ell_{i}$. As $\nabla_{u}\neq0$, we see that there is an
element in the normal form depending of some
$u_{i_{0}}^{(\lambda_{i_{0}})}$ with $\lambda_{i_{0}}-\ell_{i_{0}}=s$, and
of the form $x_{j_{0}}^{(\beta_{j_{0}}+\Lambda+s)}=g(x)$, where
$\Lambda=\max_{i}\ell_{i}$. Remaining elements in the normal form have
left side derivatives $x_{j}^{(\beta_{j}+\Lambda+s')}$, with
$s'\le s$, so $g$ does not
depend on derivatives $x_{j}^{(\beta_{j}+\Lambda+s)}$: this
implies that the minor of $\nabla_{u}$ obtained by supressing row
$i_{0}$ and column $j_{0}$ is not of full rank, a contradiction.
\end{proof}

Jacobi gave no proof for these statements; the style of the article
\cite{Jacobi2} and of most of his manuscripts on the subject is that
of a mathematical cook-book: in the best case, proofs are reduced to a
short sketch. Ritt surmised \cite{Ritt35b} that the bound was
suggested to Jacobi by such considerations on normal form
computation. This natural assumption is confirmed by Jacobi's claim
\cite{Jacobi1} that his normal form reduction method provides an
alternative proof of the bound. We easily see, according to the shape
of the normal form above, that the order of $u$ is $\sum_{i=1}^{n}
\alpha_{i}+\beta_{i}=J$.

%

Jacobi also claims that there are as many normal forms of this kind as
there are permutations $\sigma$ such that $\sum_{i=1}^{n}
a_{i,\sigma(i)}=J$. This is perhaps his  single claim that does not stand, even under
suitable genericity hypotheses, as shown by the example
$x_{1}''+x_{2}''+x_{3}''=0$, $x_{2}'=0$, $x_{2}+x_{3}=0$: we only have
one possible permutation, but two possible normal forms for shortest
reductions: $x_{1}''=-x_{2}''-x_{3}''$, $x_{2}'=0$, $x_{3}=-x_{2}$ and
$x_{1}''=-x_{2}''-x_{3}''$, $x_{3}'$=0, $x_{2}=-x_{3}$.

If we want to compute a normal form using a different kind of
orderings, we may need to differentiate the defining equations a greater
number of times. Jacobi
provided bounds for the computation of resolvents.

\subsection{Resolvent computation}
\label{resolvent}
In \S~4 of the second article \cite{Jacobi2}, Jacobi investigates the computation of \emph{
resolvents}\footnote{This word is not used by Jacobi, we borrow it from
{\it Diff.~alg.} \cite{Ritt50} p.~41}, i.e.\ of normal forms such that
one equation of order $J$,
$x_{i_{0}}^{(J)}=f_{i_{0}}(x_{i_{0}})$, only depends on a single
indeterminate $x_{i_{0}}$ and the remaining ones, $x_{i}=f_{i}(x_{i_{0}})$, $i\neq
i_{0}$, express the other variables as functions of $x_{i_{0}}$ and
its derivatives of order lower than $J$. He
writes: \emph{``As mathematicians use to consider such kind of normal forms
before others, I will indicate how many times each of the proposed
differential equations $u_{1}=0$, $u_{2}=0$, \dots $u_{n}=0$ are to be
differentiated in order to make appear auxiliary equations necessary
for that reduction.''} Jacobi's presentation requires the implicit
hypothesis $|\nabla_{u}|\neq0$, but his results stand in the more
general situation of quasi-regular systems\cite{Ollivier2009}. 
\smallskip

Let $A:=(a_{i,j})$ be defined as in the introduction, Jacobi starts
with the new matrix $A'$ defined by adding to each row of $A$ the
corresponding numbers $\ell_{i}$ defined above. He assumes that a set
of transversal maxima is chosen in the canon $A'$; he denotes
transversal maxima with asterisks and calls them \emph{stared terms},
and he underlines the terms being equal to the stared maximum located in their
column. He says that row $i_{1}$ is \emph{attached (annexa) to
row $i_{0}$} if it contains a stared term equal to some underlined
term in row $i_{0}$ or in some row attached to row $i_{0}$. The
row $i_{0}$ is implicitely assumed to be attached to itself. If not
all rows are attached to row $i_{0}$, he increases all the rows
attached to it of the same minimal number that makes new underlined
elements to appear and new rows to be
attached to row $i_{0}$. The process is to be continued until all
rows are attached to it. We get a new matrix $A''$; the last step is
to increase all its elements of a same number that makes the stared
term in row $i_{0}$ become equal to the order $J$ of the system $u$:
we get a new matrix $A^{\prime\prime\prime}$, which is obtained by
increasing the rows of $A$ by a sequence of numbers
$h_{i}$, $1\le i\le n$, that will be the researched orders of derivation.

\begin{example}
We consider a simple system $x_{1}''-x_{2}'=0$, $x_{2}''-x_{3}=0$ and
$x_{3}'-x_{2}=0$. For this system, the order matrix 
$$
A=A'=\left(\begin{array}{ccc}2^{\ast}&1\pa&-\infty\cr
  -\infty&2^{\ast}&0\cr-\infty&0\pa&1^{\ast}\end{array}\right) 
$$
is already a canon. It is not possible to construct a resolvent
representation using $x_{2}$ or $x_{3}$, but it is possible with
$x_{1}$. There is no row attached to the first one. Increasing it by $1$,
row $2$ becomes attached to row $1$. Then, increasing rows $1$ and $2$ by
$1$, row $3$ becomes attached to row $1$. So, we get the matrix
$$
A''= \left(\begin{array}{ccc}4^{\ast}&\underline{3}\pa&-\infty\cr
  -\infty&3^{\ast}&\underline{1}\pa\cr-\infty&0\pa&1^{\ast}\end{array}\right). 
$$
The order $J$ of the system is $5$ and we need to increase all terms
by 1, so that the stared term of row $1$ be made equal to $J$. We get the
new matrix
$$
A^{\prime\prime\prime}= \left(\begin{array}{ccc}5^{\ast}&\underline{4}\pa&-\infty\cr
  -\infty&4^{\ast}&\underline{2}\pa\cr-\infty&1&2^{\ast}\end{array}\right). 
$$
To obtain $A^{\prime\prime\prime}$, one has increased row $1$ of $A$
by $3$, row $2$ by $2$ and row $3$ by $1$, so one needs to
differentiate the first equation $3$ times, the second $2$ times and the
third $1$ time to compute a resolvent representation for $x_{1}$:
$x_{1}^{(5)}= x_{1}''$, $x_{2}= x_{1}^{(4)}$,
$x_{3}=x_{1}^{\prime\prime\prime}$. 
\end{example}

\begin{theorem}
Assume that $|\nabla_{u}|\neq 0$ and that $x_{j}$ is a differential
primitive element\footnote{This means that a resolvent exists for that
  element, see Cluzeau and Hubert \cite{CluzeauHubert2003} for details.}, then a
resolvent representation of $u$ using $x_{j_{0}}$ may be computed
using derivatives of equation $u_{i}$ up to order $h_{i}$.
\end{theorem}

Jacobi gave no proof for this result; we propose for the reader's
convenience the following elementary one, that follows Jacobi's
presentation. 

\begin{proof}
Using theorem~\ref{shortest-reduction}, we may assume that the system
admits a normal form of the shape
$x_{i}^{(a_{i,i})}=f_{i}(x)$, that can be obtained using
derivatives of $u_{i}$ up to order $\ell_{i}$ at most. Let us denote
by $w_{i}$ the $i^{\rm th}$ equation of this normal form. We may make
a more precise evaluation of the order of derivation requested to
compute $w_{i_{0}}$. If $x_{i}^{(\alpha_{i_{0}}+\beta_{i})}$ appears
in $u_{i_{0}}$, then we will need the $(\ell_{i}-\ell_{i_{0}})^{\rm
th}$ derivative of $w_{i}$ to compute $w_{i_{0}}$. Then, the row $i$
is attached to the row $i_{0}$ in $A'$. We may recursively prove
that, if the $(\ell_{i}-\ell_{i_{0}})^{\rm
th}$ derivative of $u_{i}$ is needed, then the row $i$
is attached to the row $i_{0}$ in $A'$. More precisely, if we need to
increase the row $i_{0}$ of $A'$ by $s\le \ell_{i}-\ell_{i_{0}}$, so that row $i$
becomes attached to it, we only
need to differentiate $u_{i}$ up to order $\ell_{i}-\ell_{i_{0}}-s$ in
order to compute $w_{i_{0}}$.

As $|\nabla_{u}|\neq 0$, the order of the system is $J$ and for
computing a resolvent, we need the first $J$ derivatives of
$x_{i_{0}}$; we shall differentiate equation $w_{i_{0}}$, which is of
order $a_{i_{0},i_{0}}$ in $x_{i_{0}}$, up to order
$J-a_{i_{0},i_{0}}$. At the beginning of the process and after each
step of differentiation, if some $x_{j}^{a_{j,j}}$, $j\neq i_{0}$
appears in the right side, we may substitute to it the expression
$f_{j}(x)$, and repeat the process until no such derivative
appears\footnote{This idea appears in some unpublished manuscripts of
Jacobi, e.~g.\ in {II/23 a)} p.~2217.a: \emph{``et simulac in dextra parte
obvenit variabilis $x_{[j]}$ differentiale [$a_{j,j}$]$^{tum}$, eius e
[$u$] substituo valorem [$f_{j}$].''} (Mathematical notations between
square brackets have been changed to correspond with those used in our
proof.) See our translation \cite{Jacobi2e} p.~37.}. If the row $i$ of $A'$ is attached to row $i_{0}$, then we
may compute the derivative of order $\ell_{i_{0}}$ of $w_{i_{0}}$
using the derivatives of $u_{i}$ of order $\ell_{i}$ at most, and so
we may compute a resolvent using the derivatives of $u_{i}$ of order
at most $J-(a_{i_{0},i_{0}}+\ell_{i_{0}})+\ell_{i}=h_{i}$. If the row
$i$ in $A'$ becomes attached to row $i_{0}$ after this last row has
been increased of $e$, then we may compute the derivative of
$w_{i_{0}}$ of order $e+\ell_{i_{0}}$ using the derivatives of $u_{i}$ up to order
at most $\ell_{i}$. So, we may get the derivative of order
$J-a_{i_{0},i_{0}}$ of $w_{i_{0}}$, requested to get the
resolvent, using derivatives of $u_{i}$ up to order  at most
$J-a_{i_{0},i_{0}}-(e+\ell_{i_{0}})+\ell_{i}=h_{i}$, hence the
result.
\end{proof}

\begin{remark}
It is easily seen that the maximal possible value for $h_{i}$ is
$J-a_{i_{0},i_{0}}+\ell_{i}-\ell_{i_{0}}=J-(\beta_{i_{0}}+\Lambda)+\ell_{i}$. If
the equation $u_{i}$ has order $e_{i}$, the sum of the $h_{i}$, that
is the number of equations in the system one needs to solve to compute
a resolvent, is maximal when $e_{i}$ is the order of $u_{i}$ in all
the variables: $h_{i}$ is at most $\sum_{i'\neq i} e_{i'}$ and
$\sum_{i=1}^{n} h_{i}\le (n-1)\sum_{i=1}^{n}e_{i}$. On the other hand,
assume that, after some reordering, the rows are listed by successive
order of ``attachment'' to row $i_{0}=1$.  We have
$h_{1}=\sum_{i>1}a_{i,i}$ and for $i>1$, $h_{i}\ge
\sum_{k>i}a_{k,k}$, so that $\sum_{i=1}^{n}
h_{i}\ge\sum_{i>1}(i-1)a_{i,i}$.
\end{remark}

\begin{remark}
In more general situations, that is when $|\nabla_{u}|= 0$, one may need to
differentiate the defining equations a greater number of times, that
may also depend on the degree of the equations. But for quasi-regular
systems \cite{Kondratieva2009,OllivierSadik2006a}, 
Jacobi's bound for the $h_{i}$ still stands\cite{Ollivier2009}.
\end{remark}
\medskip

Jacobi did not stop at this stage; he has also provided the
following elegant version of his result.

\begin{theorem}
\label{forma-elegans}
The order $h_{i}$ up to which one needs to differentiate equation
$u_{i}$ in order to compute a resolvent representation for $u$ using
$x_{i_{0}}$ as a primitive element, is equal to the maximal
transversal sum of the $(n-1)\times(n-1)$ matrix obtained by removing
from $A$ line $i$ and column $i_{0}$.
\end{theorem}

\noindent Jacobi gives a complete proof of the equality between the order
$h_{i}$ defined by his process and the maximal transversal sum of the
above theorem. It relies on the same kind of argument as the proof of
his algorithm. We refer to his paper  \cite{Jacobi2} for
details\footnote{See our translation \cite{Jacobi2e} p.~62--63.}.

We may remark that a naïve use of this theorem requires $\cO(n^{4})$
operations to compute the $h_{i}$ using Jacobi's or Munkres' algorithm  \cite{Munkres57}. But Jacobi's
process, described above, only requires $\cO(n^{3})$ operations, once
the $\ell_{i}$ have been computed with the same complexity. Jacobi states first an efficient
algorithmic version of his result, before giving an elegant---but
less efficient---mathematical version.
\medskip

These sharp bounds provided by Jacobi for computing normal forms
can be used in a straightforward way to improve many
algorithms developped in recent years. His results on normal form are
especially important for new methods of resolution for algebraic
systems relying on the representation of polynomials as \emph{
  Straight-Line Programs} \cite{GiustiLecerfSalvy1999} that
begin to be extended to differential systems
 \cite{DAlfonso2006,DAlfonsoJeronimoSolerno2006a,DAlfonsoEtAl2008,DAlfonsoEtAl2009}.

\subsection{The bound}
\label{bound-proof}
The work of Jacobi suffers from the lack of a rigorous theory allowing
precise definitions of the mathematical objects he considers. However,
if one has in mind that his goal is to consider physical situations,
that imply implicit hypotheses, his proof of the bound is not so weak
as one may think at first sight.
\medskip

\textsc{Proof of the bound.} --- The proof relies on three succesive
steps. The first is to claim that one can reduce to \emph{linear}
equations. This, of course, will not stand for all systems, but for
those expressing the laws of mechanics or other problems of physical
interest, we can take this for granted. 
\smallskip

From a mathematical standpoint, the most general condition under which
this may be done is the one given by Johnson \cite{Johnson1978},
expressing that a differential system is in some way
``regular''\footnote{This was to be developped later by Kondratieva
\emph{et al.}
 \cite{Kondratieva1982b,Kondratieva2009,OllivierSadik2006a}.}. Under 
such hypotheses, one may from a system $u_{i}=0$ build a new system
$\delta u_{i}=0$, which is nothing else than what is described by
Johnson as ``Kähler differentials'' \cite{Johnson1969}. Let
$\tilde X:t\mapsto \tilde x(t)$ be a solution of $u$. It is a
\emph{quasi-regular} solution 
of $u$ at $t_{0}$ if for all $r\in\N$, when substituting to any
derivative $x_{j}^{(k)}$ the value $\tilde x_{j}^{(k)}(t_{0})$, the
jacobian matrix of the system
$u, \ldots, u^{(r)}$ has full rank $n(r+1)$. Then the jet of $\tilde X$ at
$t_{0}$ is a regular point of the subspace $V$ of
$\bJ^{\infty}(\R,\R^{n})$ defined by $u$ and the order of $u$, that is
also the dimension of $V$, is equal to the dimension of its tangent
space, defined by $\delta u$.

It is easily seen that, if $\nabla_{u}\neq 0$ for $\tilde X$, then
$\tilde X$ is quasi-regular. We say ``quasi-regular'', for some
``singular solutions'', according to Ritt's teminology, such as $x=0$
for the equation $(x')^{2}-4x=0$, may satisfy the Johnson condition.

Jacobi uses in fact a stronger result, but without proof, claiming
that if the system $u$ has order $e$ and if its general solution
depends on $e$ arbitrary constants $a_{1}$, \dots, $a_{e}$, then the
set $\partial x/\partial a_{i}$ is a basis of solutions of the linear
system $\delta u$. I don't know if it was ``well known'' at that time.
\smallskip

The next step is more surprising, for Jacobi claims, with no
justification, that it is enough to consider a linear system with
constant coefficients. There is a beginning of proof, striked out by
Jacobi on page 2203.a of manuscript {II/13 b)}: \emph{``In exploring
the order of a system, as one considers only the highest derivatives
in the linear differential equations to which the proposed ones have
been reduced, one may assume that the coefficients are constants. For,
having differentiated the equations [$\delta u=0$] many times, in
order to obtain new equations''}\dots\ This interrupted sentence
suggests that the shortest reduction process \cite{Jacobi2} used to
compute normal forms was the basic idea. 

In fact, provided that
$|\nabla_{u}|$ does not vanish, we only need to consider the
coefficients of the leading derivatives\footnote{We call here a
leading derivative of an equation, a derivative $x_{j}^{(k)}$ such
that $k-\beta_{j}$ is maximal, i.e.\ with a maximal Jacobi order.}
when computing the order: derivatives of the coefficients will only
affect smaller derivatives of the variables. If $|\nabla_{u}|$
vanishes, assume that the equations $u_{i}$ are sorted by increasing
$\alpha_{i}$ and that $i_{0}$ is the smallest integer such that the
first $i_{0}$ rows of $\nabla_{u}$ are linearly dependent, satifying
$\sum_{i=1}^{i_{0}} c_{i}l_{i}=0$, then we may in $\delta u$ replace
$\delta u_{i_{0}}$ by $\sum_{i=1}^{i_{0}-1} c_{i} (\delta
u_{i})^{(\ell_{i}-\ell_{i_{0}})}$, as $c_{i_{0}}\neq0$, this new
system is equivalent to $\delta u$ and the leading derivatives in
$\delta u_{i_{0}}$ have been removed: so the new system has a strictly
smaller Jacobi number $J$. This gives an easy proof by induction of the
result.
\smallskip

We could conclude from such considerations on normal forms of linear
systems, but Jacobi uses a different kind of argument for the last
step of his proof. Having reduced his investigations to the case of a
linear system with constant coefficients, he looks for solutions of
the form $x_{j}=c_{j}{\rm e}^{\lambda t}$. Substituting such an
expression in his linear equations, he gets a system of the form
$\sum_{j=1}^{n} P_{i,j}(\lambda)\,c_{j}=0$, $1\le i\le n$ where
$\deg_{\lambda} P_{i,j}=a_{i,j}$. The number of possible values for
$\lambda$ is the degree of the determinant $|P|$, which is at most
$\max_{\sigma\in S_{n}} \sum_{i=1}^{n} a_{i,\sigma(i)}=J$, with
equality if $|\nabla_{u}|\neq0$. Jacobi does not consider the case of
multiple solutions, etc.\ but there is no difficulties.
\medskip

We cannot know precisely how Jacobi could have detailed his
demonstration. However, we have shown that we can design a complete
proof, using elementary arguments, following the indications he left
in his manuscripts, provided that we retreat to the safe ground of
regular systems. For an account of the research on singular solutions
during the first half of the \textsc{xix}$^{\rm th}$ century, see the
work of Houtain  \cite{Houtain1852}.

\section{The second part of the \textsc{xix}$^{\bf th}$ century}

The first publication of these two papers\cite{Jacobi1,Jacobi2} in
1865 and 1866 and a new publication in 1890 in the volume~V of
Jacobi's complete works did not stimulate further research on the
subject in Germany. The works of his continuators during the
\textsc{xix}$^{\rm th}$ century are very superficial. They did not
seem to have considered the subject could deserve a real mathematical
effort.
\smallskip

Nanson in 1876  \cite{Nanson1876}, considers the linear case
with constant coefficients, in $2$ and $3$ variables. He rejects the
idea of substituting ${x_{i}=c_{i}\rm e}^{\lambda t}$ in the equation,
in order to obtain the order, claiming that one should first compute
the order in order to be sure that a complete solution of that kind
could be obtained. He proceeds heuristically, eliminating one variable
after the other and bounding at each step the orders in each variable
of the equations he gets.
\smallskip

Jordan in 1883  \cite{Jordan1883} considered the non linear situation with
$4$ variables.  He tried to eliminate $x_{2}$, $x_{3}$, $x_{4}$ in
order to compute a resolvent for $x_{1}$, using
arguments that only work in the most general situation. From
heuristic considerations on the number of derivatives to eliminate,
he established theorem \ref{forma-elegans} in four variables: one needs
to differentiate $u_{i}$ a number of times
 equal to $h_{i}$, the maximal
transversal sum of the matrix obtained by removing from the order
matrix $A$ row $i$ and column $1$. So, the order of the resolvent
will be $\max_{i=1}^{4} a_{i,1}+h_{i}=J$. 
\smallskip

Nanson, who referred to Boole \cite{Boole1859} for systems of order $1$ and
Cournot \cite{Cournot1857} for systems of two equations, did
not quote Jacobi. Jordan---who did not quote Nanson---did
not have a full view of Jacobi's work.  He claimed that Jacobi
had given an ``indirect'' proof and that he would give a ``direct'' one. 
\smallskip

The work of Chrystal in 1895  \cite{Chrystal1895} was rigorous, but he
only considered the easy linear case with constant coefficients---Jacobi's
arguments only worked for all different eigenvalues. Ritt, who gave
these references  \cite{Ritt35b} also referred to a paper by Sarminski
(\emph{Communications of the University of Warsaw}, 1902) that I was
unable to find.

\section{Ritt's work}
\label{Ritt}

Ritt, who was known to be fluent in many languages, certainly had a
better view of Jacobi's results. However, a century after Jacobi wrote
them, the style and spirit of mathematics did change. One expects
rigorous proofs, but also intrinsic results, attached to geometrical
objects. One thinks of varieties (or ``components'') and not of
systems: more precisely, in his article \cite{Ritt35b}, published in 1935, a
system means a component. It is remarkable that this change also
concerned a mathematician like Ritt, who knew well some very applied
style of mathematics\footnote{He worked performing computations in the
Naval Observatory in Washington during his studies and helped to
organise a computation group working for the US artillery during World
War~I \cite{Grier2001}.} and whose activity was dominated by the
spirit of ``classicism''\footnote{In Ritt's obituary \cite{Lorch1951} p.~310,
E.R.~Lorch writes \emph{``His media are complex function theory of the
nineteenth century and differential equations. Much of his work could
have been written a half century earlier.''}}.

One interest of Ritt in this subject was to secure results that could
be applied to components intersection and it is not the case for
Jacobi's bound (see \emph{Diff.~Alg.} \cite{Ritt50} p.~138--144). Having developped a
theory that allows to characterize ``singular components'', he
wonders if the bound stands for all components, including those that
do not satisfy natural hypotheses of regularity: a difficult question
that was certainly not considered by Jacobi.
\smallskip

The less convincing part of Jacobi's ``proof'' is to go from time
varying systems to constant coefficients.  Ritt's proof in the linear
case  \cite{Ritt35b} solves the problem. But Ritt, who only considers
 \cite{Ritt35b,Ritt50} elimination orderings does not prove the necessary
and sufficient condition for the bound to be reached, given by the non
vanishing of the truncated Jacobian. Such a condition is, as we have
seen in \ref{bound-proof}, more easily proved with an orderly ordering
for an adapted order defined in this way:
$\widetilde{\ord}_{x_{j}}u_{i}:=\ord_{x_{j}}u_{i}-\beta_{j}$, with
$\beta_{j}$ defined as in def.~\ref{truncated}. Ritt's proof relies on
some simplified method for computing a characteristic set in the
linear case. He proves the \emph{strong bound}, using the convention
$\ord_{x_{j}}u_{i}=-\infty$ if $x_{j}$ and its derivatives do not
appear in $u_{i}$.
\medskip

His 1935 paper  \cite{Ritt35b} concludes with a proof of
the bound for any component of dimension zero of a system of two
polynomial equations $A$ and $B$ in two variables, that is reproduced with a few
modifications in \emph{Diff. alg.} \cite{Ritt50} p.~136--138. One may remark that in
the 1935 article \cite{Ritt35b}, a footnote precises that if one of the variables does
not appear in $A$, the order of $A$ in this variable is $0$. In
 \cite{Ritt35b}, a new footnote claims that we can in fact prove the
strong bound. The requested modifications in the proof seem easy, but
are not given explicitely. In 1935 \cite{Ritt35b}, Ritt refers to Gourin
 \cite{Gourin1933} for the following result: if a zero dimensional
differential ideal $\cI$ contains a zero dimensional differential
ideal $\cJ$, then $\ord\,\cI<\ord\,\cJ$. 

An important argument is not explicitly stated in the proof. Ritt
reduces the situation to the case of two polynomials, $A$ and $E$,
were $E$ depends only on $x_{1}$. He claims that $E$ must effectively
depend on that variable. For this, he needs to use the fact that a
single equation in two variables cannot define a component of
dimension $0$, which is true by \emph{Diff. alg.} \cite{Ritt50} chap.~III \S~1 p.~57.
\medskip

In \emph{Diff.~alg.}\ \cite{Ritt50} p.~139--144, Ritt also proved an important result. He considers
irreducible differential polynomials in two variables, and
investigates the order of the intersection of their \emph{general
  solutions}. He first proves that Jacobi's bound stands if $A$ and
$B$ have order not greater than unity. Then, he exhibits a family of
polynomials of order $r>3$ in $x_{1}$ and $x_{2}$, the general solution
of which intersects the manifold of $x_{1}$ in an irreducible manifold
of order $2r-3$. So, the bound cannot stand for manifolds, but just
for systems.

\section{The assignment problem}
\label{assignment}

In 1944, the R.A.F. tried to optimize the reaffectation of soldiers of
disbanded units \cite{Schrijver05}. No practical solution could be used
before the end of World War~II, but this initiated the first research
on the problem. It was then considered to optimize the affectation of
$n$ workers to $n$ tasks, $a_{i,j}$ representing the productivity of
worker $i$ if affected to task $j$. One looks for a maximum, with the
constraint that two different workers must be given two different
tasks. The Monge problem
 \cite{Monge1781} may be considered as a first, continuous, example
of such problems (how to transport earth from a given area to some
other with the least amount of carriage).
\smallskip

It is not the place to give much details about the discovery of the
Hungarian method by Harold Kuhn in 1955  \cite{Schrijver05}. Anyway, it
may be of interest to consider the situation from the standpoint of
the transmission of mathematical results. Jacobi's algorithm was sleeping in
papers written in a dead language, with titles that cannot be related
to the assignment problem. It also seems that the mathematical
community was not always of a great help for the practitioner who
wanted to solve his optimization problem in a short time. It is
amazing that trying $n!$ possibilities may have been
considered as an acceptable solution, provided that
their number is finite, in the middle of the
\textsc{xx}$^{\rm th}$ century (see Schrijver \cite{Schrijver05}
p.~8). For Jacobi, trying $n!$ solution was not a 
solution at all. He claimed indeed to look for \emph{a solution},
whereas we would rather say that we are looking for an efficient
one. The efficiency issue was at that time---very strongly---implicit.

One may also notice that rediscovering Jacobi's method took more than
10 years, from 1944 to 1955, and that prominent mathematicians such as
John von Neuman considered the problem. It could have been much longer
if Kuhn did not translate from Hungarian Egerváry's paper
\cite{Egervary31} that allowed him to conclude. Inspired by K{\H
o}nig, a pioneer of graph theory\cite{Konig50}, Egerváry considered
the problem as a weighted variant of the maximal matching problem, but
he did not give a polynomial time algorithm
\cite{Juttner2004}. Possibly, Egerváry could have
contributed to the question himself, but it seems that the research
was concentrated in the eastern part of the world, mostly in the
United States. It was also strongly motivated by economical and
organizational issues and one may guess that it did not facilitate
collaborations with eastern scientists during the cold war. Egerváry
heard of Kuhn's algorithm as late as in 1957 and went back to such
matters with two papers on the transportation problem. Tragic
circumstaces interrupted his research. If, after the war, the
Hungarian Academy of Sciences provided him good working conditions
that stimulated his work, he killed himself in 1958, persecuted by
bureaucrats\footnote{K{\H o}nig commited suicide in october 1944, a
few days before Budapest Jews were forced into the ghetto and their
deportation began.}.
\smallskip

Richard Cohn is the first, for the best of my
knowledge, to have made a link between Jacobi's work and the
assignment problem  \cite{Cohn1983}, but the information did not spread into the
optimization community before 2005.

\section{The second part of the \textsc{xx}$^{\bf th}$ century}

The second part of the \textsc{xx}$^{\rm th}$ century is dominated by the work
of Richard Cohn and his students. 

\subsection{Greenspan's bound} Greenspan proved a different bound,
in the framework of difference algebra  \cite{Greenspan1959b} in
1959. It is easily translated in differential algebra: let
$r_{j}=\max_{i=1}^{n}\ord_{x_{j}}u_{i}$ and $\eta_{i}$ be the greatest
integer such that $\ord_{x_{j}}u_{i}^{(\eta_{i})}\le r_{j}$, $1\le j\le
n$, Greenspan's bound is $G:=\sum_{i=1}^{n} r_{j}-\max_{i=1}^{n}
\eta_{i}$. It was proved by Cohn in 1980  \cite{Cohn1980} that the order
of any zero dimensional component of an arbitrary differential system
is bounded by $G$. One may remark that this result implies Jacobi's
bound in two variables, and that it may be proved using an adapted
version of Ritt's proof.

\subsection{Lando's bound}
Barbara Lando proved in 1970  \cite{Lando70} the
\emph{``weak bound''} for order one differential systems. The \emph{``weak bound''}
means that if $x_{j}$ and its derivatives do not appear in $u_{i}$, we
use the convention $\ord_{x_{j}}u_{i}=0$.  This result was translated
in difference algebra  \cite{Lando72} in 1972. \emph{Proving the strong
bound for order one systems would imply the strong bound for any
system, for the strong bound is compatible with the classical
reduction of a system to order one equations.} Lando's proof uses
results on matrices of zeros and ones that are very close to some
theorems of K{\H o}nig and Egerv{\' a}ry.

B.~Lando also proved that any order matrix $A$ is the order matrix of
a system for which the bound is reached.

\subsection{Tomasovic and PDE systems}
There was a attempt to generalize the bound to partial differential
systems  \cite{Tomasovic76}, due to Tomasovic in 1976. Consider a
system of $n$ partial differential equations in $n$ variables and $m$
derivatives. Let $\cP$ be a component of of $\{u\}$ and
$$
\omega_{\cP}(r) = \sum_{i=0}^{m} a_{i} {r+i\choose i}
$$
be the Hilbert polynomial of $\cP$. If the dimension of $\cP$ is $0$,
then $a_{m}=0$. The \emph{Jacobi conjecture} of Tomasovic states
that $a_{m-1}\le J$, where $J$ is defined as above, according to the
order matrix of the PDE system. This conjecture had already been
presented by his thesis adviser, Kolchin, in 1966
 \cite{Kolchin1966}. Tomasovic proved it for linear systems and for
$n\le 2$. A proof in the linear case was also given by
Kondratieva \emph{et al.} \cite{Kondratieva1999}.
\smallskip

Tomasovic's results remained unpublished, due to his
untimely death. His dissertation contains interesting material that
requires further examination.

\subsection{Order and dimension}
In 1983, Cohn proved that the bound would imply the \emph{``dimension
conjecture'': every component defined by a system of $r$ equations has
differential codimension at most $r$}. Tomasovic proved in the chap.~6
of his thesis \cite{Tomasovic76} that the dimension conjecture is
equivalent to this one: \emph{If a system has a component of
differential dimension $0$, then its Jacobi number $J$ is not
$-\infty$.}

We have seen in section~\ref{Ritt} that Ritt's proof in two variables requires
the dimension conjecture for $r=1$. More precisely, Cohn's proof shows
that proving the bound for a system of $n$ equations implies the
dimension conjecture for a system of equations in $r<n$ variables.
Furthermore, Cohn proved that the weak bound and the dual of
``Bézout's bound''\footnote{It is defined as $\sum_{i=1}^{n}
\max_{j}\ord_{x_{j}} u_{i}$.} also imply the dimension conjecture.
\smallskip

It is known that the intersection of two manifolds of differential
codimensions $r$ and $s$ may contain components of codimension greater
that $r+s$ (see \emph{Diff.~alg.} \cite{Ritt50} p.~133) and, as we have already seen
above in section\ref{Ritt}, Jacobi's bound may only be expected to stand for
systems and not for manifolds. If the examples given by Ritt in these
two cases are clearly distinct, we may remark that they are both closely
related to the structure of the singular place of the manifold. We still
need to understand better such paradoxical behaviours and their
possible connections.

\subsection{Kondratieva's proof}
As we have already seen in subsection \ref{bound-proof}, the case
where Jacobi's linearization argument works corresponds to the
regularity hypothesis defined by Johnson in order to prove Janet's
conjecture  \cite{Johnson1978} in 1978: \emph{the differentials
$du_{i}^{(k)}$, $1\le i\le n$, $k\in\N$, are linearly independent.}
Kondratieva \emph{et al.} call such systems \emph{independent systems};
they were able in 1982 \cite{Kondratieva1982b} to
prove the strong bound, using first linearization as in Jacobi's approach
 \cite{Jacobi1}, then Ritt's proof for the linearized system
$du$. However, this result received little attention, the paper being
written in Russian and difficult to find. A new proof, also valid for
independent partial differential systems has been given in 2008 by the same
authors \cite{Kondratieva2009}.

\subsection{Other works}
In 1960, Jacobi's strong bound was rediscovered independently by
Volevich  \cite{Volevich1960} for arbitrary linear systems,
assuming that the truncated determinant does not vanish. One may also
mention the works of Magnus,  \cite{Magnus2001,Magnus2003}
who refers to Chrystal and Jacobi.

\section{Beginning of the \textsc{xxi}$^{\bf th}$ century}

Hrushovski in 2004  \cite{Hrushovski2004} proposed a proof for the
bound in difference algebra. His method is completely different
from those used so far in this field, but it does not seem possible,
or at least not easy, to deduce from this result a proof in the
differential case. 

See our article \cite{OllivierSadik2006a} for a proof of the truncated jacobian
condition in the framework of diffiety theory, assuming the regularity
condition of Johnson. A generalization to underdetermined systems is
also considered. Kondratieva {\it et al.} \cite{Kondratieva2009}
provided a proof of Jacobi's
bound for independent partial differential systems.

In 2001, Pryce \cite{Pryce2001} rediscovered Jacobi's shortest reduction
method in order to provide a efficient method of computing power
series solution of implicit differential algebraic systems.

\section{Conclusion}

One says that Jacobi once told to a student who wanted to read all
the mathematical literature before starting his research: \emph{``Where
would you be if your father before marrying your mother had wanted to
see all the girls of the world?''} So, we may hope that he would have
forgiven us for having forgotten some of his results.

We see nevertheless that a closer look to the past may
be fruitfull. Contemporary mathematicians who turn to computer algebra
will find some common spirit with the approach of these times, when a
great familiary with hand computation produced a still unformal but
deep attention to efficiency.

Some of the results presented here, beyond their  intrinsic
mathematical interest, could help producing improved bounds of
complexity or designing new algorithms. E.g., the shortest
reduction leads to the choice of an ordering for which the
computation of a characteristic set may be easier.

\section*{Thanks}
I express my gratitude to the late Evgeny Pankratiev, Marina
Kondratieva, Alexandr Mihailev, Brahim sadik and the referees for
their carefull rereading and many corrections.

Thanks to Richard Cohn, Marina Kondratieva, Harold Kuhn, William Sit
for scientific comments and historical precisions. I also express my
thanks to Dr~Wolfgang Knobloch and Dr.~Vera Enke (Archiv der BBAW)
for their precious help in my search for Jacobi's manuscripts, to
Bernd Bank for achieving the deciphering of Cohn's letter
{II/13 a)}, to mgr Ivo {\L}aborerevicz (Archiwum
Pa{\'n}stwowe we Wroc{\l}awiu) and mgr Marlena Koter
(Archiwum Pa{\'n}stwowe w Olsztynie), to Jean-Marie Strelcyn
for his kind provinding translations of letters from the polish
archives, to Bärbel Mund (Niedersächsische Staats- und
Universitätsbibliotek Göttingen), Mikael Ragstedt (library of
the Mittag-Leffler Institute). 

Last but not least, I express my gratitude to all the staff of the
Archives and Bibliothèque Centrale de l'École polytechnique.

The ``Groupe Aleph et GÉode'' provided financial
support for paying copies of the documents.
\bigskip

\section*{Warning}
{\sf The bibliography is divided in four parts. The first, \emph{
  Manuscripts} contains primary material, mostly Jacobi's manuscript,
  comming from \emph{Jacobis Nachla{\ss}, Archiv der Berlin-Brandenburgische
  Akademie der Wissenschaften}. Documents are denoted by their archive
  index, e.g. {\rm [I/58 a)]}. The second, \emph{Complete works} contains the
  book were Jacobi's work were published. They are denoted by
  {\rm [GW \dots], [VD]}, \dots\ The third \emph{Crelle Journal} contains the
  issues were the quoted papers of Jacobi were published. Their are
  denoted by {\rm [Crelle\dots]} The fourth and last contains the remaining
  material. They are denoted by numbers.}

\bibliographystyle{ws-procs9x6}

\end{document}